\documentclass[12pt, reqno]{amsart}
\usepackage{amsmath, amsthm, amscd, amsfonts, amssymb, graphicx, color}
\usepackage[bookmarksnumbered, colorlinks, plainpages]{hyperref}
\hypersetup{colorlinks=true,linkcolor=red, anchorcolor=green, citecolor=cyan, urlcolor=red, filecolor=magenta, pdftoolbar=true}

\newtheorem{theorem}{Theorem}[section]
\newtheorem{lemma}[theorem]{Lemma}

\theoremstyle{definition}
\newtheorem{definition}[theorem]{Definition}

\theoremstyle{remark}
\newtheorem{remark}[theorem]{Remark}
\numberwithin{equation}{section}

\begin{document}

\setcounter{page}{1}

\title[Nonlinear Sobolev type equations  on the Heisenberg Group]{Instantaneous blow up solutions for nonlinear Sobolev type equations  on the Heisenberg Group}

\author[M. B. Borikhanov, M. Ruzhansky and B. T. Torebek]{Meiirkhan B. Borikhanov, Michael Ruzhansky and  Berikbol T. Torebek}

\address{{Meiirkhan B. Borikhanov \newline Khoja Akhmet Yassawi International Kazakh--Turkish University \newline Sattarkhanov ave., 29, 161200 Turkistan, Kazakhstan 
\newline and
\newline Institute of
Mathematics and Mathematical Modeling \newline 125 Pushkin str.,
050010 Almaty, Kazakhstan}}
\email{meiirkhan.borikhanov@ayu.edu.kz}
\address{{Michael Ruzhansky  \newline Department of Mathematics: Analysis, Logic and Discrete Mathematics \newline Ghent University, Belgium 
\newline and
\newline School of Mathematical Sciences \newline Queen Mary University of London \newline United Kingdom}}
\email{michael.ruzhansky@ugent.be}
\address{{Berikbol T. Torebek \newline Institute of
Mathematics and Mathematical Modeling \newline 125 Pushkin str.,
050010 Almaty, Kazakhstan \newline and \newline Department of Mathematics: Analysis, Logic and Discrete Mathematics \newline
Ghent University, Belgium}}
\email{{berikbol.torebek@ugent.be}}


\subjclass[2020]{35R03, 35B33, 35B44.}

\keywords{nonlinear Sobolev type equations, instantaneous blow-up, Heisenberg group.}

\begin{abstract}In this paper, we study the nonlinear Sobolev type equations on the Heisenberg group. We show that the problems do not admit nontrivial local weak solutions, i.e. “instantaneous blow up" occurs, using the nonlinear capacity method. Namely, by choosing suitable test functions, we will prove an instantaneous blow up for any initial conditions $u_0,\,u_1\in L^q(\mathbb{H}^n)$ with $q\leq \frac{Q}{Q-2}$.  \end{abstract}
\maketitle 

\section{Introduction}

We consider the following nonlinear Sobolev type equations on the Heisenberg group
\begin{equation}\label{1}
\left\{\begin{array}{l}\large\displaystyle
\frac{\partial}{\partial t}\Delta_\mathbb{H}u+\Delta_\mathbb{H}u+|u|^{q}=0,\,\ (t,\eta)\in(0,T)\times\mathbb{H}^{n},\\{}\\
u(0,\eta)=u_0(\eta),\, \eta\in\mathbb{H}^{n},\end{array}\right.\end{equation}and 
\begin{equation}\label{2}
\left\{\begin{array}{l}\large\displaystyle
\frac{\partial^2}{\partial t^2}\Delta_\mathbb{H}u+\Delta_\mathbb{H}u+|u|^{q}=0,\,\ (t,\eta)\in(0,T)\times\mathbb{H}^{n},\\{}\\
u(0,\eta)=u_0(\eta),\,\frac{\partial}{\partial t}u(0,\eta)=u_1(\eta),\, \eta\in\mathbb{H}^{n},\end{array}\right.\end{equation}where $q>1$ and $\Delta_\mathbb{H}$ is the sub–Laplacian operator on the Heisenberg group $\mathbb{H}^{n}$. 

Many authors have studied the issue of instantaneous blow-up. Let us start by noting some well-known problems related to \eqref{1} and \eqref{2}.

Firstly, Brezis and Cabr\'{e} in \cite{Brezis1} investigated the  nonexistence of local solutions to singular nonlinear parabolic and elliptic equations, respectively
\begin{equation}\label{5}\left\{\begin{array}{l}
u_t-\Delta u\geq|x|^{-2}u^2\,\,\,\,\text{in}\,\,\,(0,\infty)\times(\Omega\setminus\{0\}),  \\{}\\
u(0,x)=u_0(x) \,\,\,\,\text{in}\,\,\,\Omega, \end{array}\right.\end{equation}and
\begin{equation}\label{4}\left\{\begin{array}{l}
-\Delta u\geq|x|^{-2}u^2\,\,\,\,\text{in}\,\,\,\Omega\setminus\{0\},  \\{}\\
u=0 \,\,\text{on}\,\,\,\partial\Omega, \end{array}\right.\end{equation}
where $\Omega\subset \mathbb{R}^N$  is a smoothly bounded domain.

Later, in \cite{Weissler, Brezis} the authors considered the question of local existence for a semilinear heat equation
\begin{equation}\label{3}\left\{\begin{array}{l}
{{u}_{t}}-\Delta u= \vert u \vert^{p-1}u\,\,\,\,\,\text{in}\,\,\,(0,T)\times\Omega,  \\{}\\
u=0 \,\,\,\,\text{on}\,\,\,(0,T)\times\partial\Omega,\\{}\\
u\left(0,x\right)=u_0\left( x \right)\,\,\,\, \text{in}\,\,\,\, \Omega, \end{array}\right.\end{equation}
where $\Omega\subset \mathbb{R}^N$  is a smoothly bounded domain and $p>1$.

Then, for the linear parabolic equation with singular potential
\begin{equation}\label{30}\left\{\begin{array}{l}
{{u}_{t}}-\Delta u= a(x)u\,\,\,\,\,\text{in}\,\,\,(0,T)\times\Omega,  \\{}\\
u=0 \,\,\,\,\text{on}\,\,\,(0,T)\times\partial\Omega,\\{}\\
u\left(0,x\right)=u_0\left( x \right)\,\,\,\, \text{in}\,\,\,\, \Omega, \end{array}\right.\end{equation}
where $\Omega\subset \mathbb{R}^N$  is a smooth bounded domain and $a(x)\geq0$, the “instantaneous blow up" phenomenon was obtained in \cite{Cabre}.

It is important to note that the comparison method has been used to establish the absence of solutions in the above-mentioned relevant references. On the other hand, standard comparison principles tend to fail for higher-order operators with respect to time. For this reason, many authors have used the nonlinear capacity method to overcome this difficulty (see \cite{Mitidieri} and references therein). This method can be used to obtain results under sufficient conditions for the nonexistence of solutions of higher-order equations in time. In particular, it was used in \cite{Galakhov, Galakhov1} to investigate the instantaneous blow-up for a variety of nonlinear evolution equations, including higher order quasilinear evolution partial differential equations and inequalities with singular coefficients.

In \cite{Coleman}, the instantaneous blow-up in nonclassical Sobolev equations on a strip
\begin{equation}\label{40}\left\{\begin{array}{l}
u_t=u_{xx}-u_{xtx},\\{}\\
u(0,x)=u_0(x),\,\,u(t,0)=u(t,l)=0,\,\,l>0, \end{array}\right.\end{equation} was discovered for the first time.
The main result about nonexistence of a bounded solution to this problem for an infinitely small period of time was established under the condition  $l\in(0,\pi]$.

Moreover, for several classes of nonlinear Sobolev type equations, including problems \eqref{1} and \eqref{2} on $\mathbb{R}^n$, the instantaneous blow up was studied by Korpusov in \cite{Korpusov}. Namely, it was shown that, if $u_0, u_1\in L^q(\mathbb{R}^n)$ and 
\begin{equation}\label{7}
1<q\leq q_c:=\left\{\begin{array}{l}
+\infty,\,\,\,\,\,\,\text{if}\,\,\,\,\,n=1,2,\\{}\\
{\large\displaystyle\frac{n}{n-2}},\,\,\,\text{if}\,\,\,\,n\geq3,\end{array}\right.\end{equation}
then the problem \eqref{1} and \eqref{2} admits no nontrivial local weak solution. In addition, we have to note that the author studied instantaneous blow up solutions in  \cite{Korpusov1}-\cite{Korpusov3}.

 Recently, Jleli et al. investigated the absence of local weak solutions for the problem \eqref{2} with potentials defined on complete noncompact Riemannian manifolds  \cite{Jleli}, and  some differential inequalities of Sobolev type in an exterior domain \cite{Jleli1}.

In general, the Sobolev type problems such as \eqref{1}-\eqref{2} have been widely studied on $\mathbb{R}^N$ in \cite{Jleli1,Korpusov,Korpusov1} and  on complete noncompact Riemannian manifolds in  \cite{Jleli}.

We also note that the Fujita-type exponents to semilinear evolution equations  on the Lie groups were studied in \cite{Kassymov, Pascucci, Pohozaev, Zhang}, and on sub-Rimannian manifolds and general unimodular Lie groups in \cite{Ruzhansky1}.

Motivated by the above works, the main objective of this paper is to obtain the instantaneous blow up phenomenon
on the Heisenberg group $\mathbb{H}^n,\,n\geq 1$.

Our main results are the following.

\begin{theorem}\label{T1}
Let $u_0\in L^q(\mathbb{H}^n)$. If $$\displaystyle 1<q\leq q_c=\frac{Q}{Q-2},$$ then the problem \eqref{1} has no local weak solution for any $T>0$, i.e., there is an instantaneous blow-up of the local weak solution of the Cauchy problem \eqref{1}.  
\end{theorem}
\begin{theorem}\label{T2}
Let $u_0\in L^q(\mathbb{H}^n)$ and $u_1\in L^q(\mathbb{H}^n).$ If $$\displaystyle 1<q\leq q_c=\frac{Q}{Q-2},$$ then the problem \eqref{2} admits no nontrivial local weak solution for any $T>0$, i.e., there is an instantaneous blow-up of the local weak solution of the Cauchy problem \eqref{2}. 
\end{theorem}

\begin{remark}
Note that the critical exponent $q_c=\frac{Q}{Q-2}$ is optimal. Indeed, if $$q>q_c=\frac{Q}{Q-2},$$ then the problem \eqref{1} and the problem \eqref{2} have stationary supersolutions, i.e., there is a function $v(\eta)>0,$ which is a positive solution to the equation (see \cite{Birindelli})
$$\Delta_\mathbb{H}v+v^{q}=0,\,\eta\in\mathbb{H}^{n},$$
for $q>q_c=\frac{Q}{Q-2}.$

Hence, for sufficiently small initial data, the function $v$ is a supersolution of problems \eqref{1} and \eqref{2}, i.e. $$u(t, \eta)\leq v(\eta),\,t\geq 0,\,\eta\in\mathbb{H}^{n}.$$ This means that the solutions to problems \eqref{1} and \eqref{2} will be bounded for all $t\geq 0$ and $\eta\in\mathbb{H}^{n}.$
\end{remark}

\section{Preliminaries}
\subsection{The Heisenberg group} In this subsection let us briefly recall the necessary notions in the setting of the Heisenberg group. For more details we refer to \cite{Fisher},\cite{Ruzhansky}.

The Heisenberg group $\mathbb{H}^{n}$, whose elements are $\eta=(x, y, \tau)$, where $x, y\in \mathbb{R}^n$ and $\tau\in \mathbb{R}$, is a two-step nilpotent Lie group $(\mathbb{R}^{2n+1}, \circ)$ with the group multiplication defined by
\begin{equation*}\label{HMF}
\eta\circ\eta'=(x+x', y+y', \tau+\tau'+2(\langle x,y'\rangle-\langle x',y\rangle)),
\end{equation*}
with $\langle \cdot,\cdot\rangle$ the usual inner product in $\mathbb{R}^{n}$.

The distance from $\eta=(x, y, \tau)\in \mathbb{H}^{n}$ to the origin is given by
\begin{equation}\label{HDO}|\eta|_{\mathbb{H}}=\left(\left(\sum_{i=1}^{n}(x_i^2+y_i^2)\right)^2+\tau^2\right)^{1/4}=\left(\left(|x|^2+|y|^2\right)^2+\tau^2\right)^{1/4}.\end{equation}

The dilation operation on the
Heisenberg group with respect to the group law is
\begin{equation*}\label{HD}
\delta_\lambda(\eta)=(\lambda x, \lambda y, \lambda^2\tau)\,\,\,\text{for}\,\,\,\lambda>0,    
\end{equation*}
whose Jacobian determinant is $\lambda^Q$, where $Q=2n+2$ is the homogeneous dimension of $\mathbb{H}^{n}$.

The Lie algebra $\mathfrak{h}$ of the left-invariant vector fields for $1\leq i\leq n$ on the Heisenberg group $\mathbb{H}^{n}$ is spanned by
\begin{equation*}\label{LIVF}\begin{split}
X_i:=\frac{\partial}{\partial x_i}+2y_i\frac{\partial}{\partial \tau},\\
Y_i:=\frac{\partial}{\partial y_i}-2x_i\frac{\partial}{\partial \tau},
\end{split}
\end{equation*}
with their (non-zero) commutator
\begin{equation*}\label{NZC}
[X_i,Y_i]=-4\frac{\partial}{\partial \tau}.    
\end{equation*}

The sub-Laplacian is defined by
\begin{equation*}\begin{split}\label{HLO}\Delta_{\mathbb{H}}&=\sum_{i=1}^n\left(X_i^2+Y_i^2\right).
\end{split}\end{equation*}

An explicit computation gives the expression
\begin{equation}\label{HLO1}\begin{split}\Delta_{\mathbb{H}}&=\sum_{i=1}^n\left(\partial^2_{x_ix_i}+\partial^2_{y_iy_i}+4y_i\partial^2_{x_i\tau}-4x_i\partial^2_{y_i\tau}+4(x_i^2+y_i^2)\partial^2_{\tau\tau}\right)\\&=\Delta_{(x,y)}+4|(x,y)|^2\partial^2_{\tau\tau}+4\sum_{i=1}^n\left(y_i\partial^2_{x_i\tau}-x_i\partial^2_{y_i\tau}\right),\end{split}\end{equation} where $\Delta_{(x,y)}$ and $|(x,y)|^2$ denote the Laplace operator and the Euclidian norm of $(x,y)$ in $\mathbb{R}^{2n}$, respectively.

Moreover, if $u(\eta)=\varphi(|\eta|_{\mathbb{H}})$, then
\begin{equation}\label{DIS}
\Delta_{\mathbb{H}}\varphi(r)=\frac{\sum_{i=1}^{n}(x_i^2+y_i^2)}{r^2}\biggl(\frac{d^2\varphi}{dr^2}+\frac{Q-1}{r}\frac{d\varphi}{dr}\biggr),    
\end{equation}where $r=|\eta|_{\mathbb{H}}$, $Q=2n+2$ is the homogeneous dimension of $\mathbb{H}^{n}$.

The operator $\Delta_{\mathbb{H}}$ satisfies the following fundamental properties:
\begin{description}
 \item[(a)] It is invariant with respect to the left multiplication in the group: for all $\eta, \eta'\in\mathbb{H}^{n}$ we have
 $$(\Delta_{\mathbb{H}, \eta'}u)(\eta\circ\eta')=(\Delta_{\mathbb{H} }u)(\eta\circ\eta').$$
\item[(b)] It is homogeneous with respect to the dilation in the group: for all $\lambda>0$ we have
 $$\Delta_{\mathbb{H}}(u(\lambda x, \lambda y, \lambda^2\tau))=\lambda^2(\Delta_{\mathbb{H}}u)(\lambda x, \lambda y, \lambda^2\tau).$$
\end{description}

Since $(-\Delta_{\mathbb{H}})$ is a self-adjoint operator, for all $u, v\in H^2(\mathbb{H}^n)$
\begin{equation}\label{IBP}
\int_{\mathbb{H}^n}(-\Delta_{\mathbb{H}}u)(\eta)v(\eta)d\eta= \int_{\mathbb{H}^n}u(\eta)(-\Delta_{\mathbb{H}}v)(\eta)d\eta.   
\end{equation}

\subsection{Test functions} In this subsection, some test functions and their properties will be considered. In addition, we shall establish some estimates involving the test functions.

We consider the test function to be such that
\begin{equation}\label{test1}
\varphi(t,\eta)=\varphi_1(t)\varphi_2(\eta)=\left(1-\frac{t}{T}\right)^\ell\Phi
\left(\frac{|\eta|_{\mathbb{H}}^2}{R^2}\right),\,\,\,t>0,\,\eta\in\mathbb{H}^n,\end{equation}where $ T,R>0, \ell
>\frac{q+1}{q-1}$ and $|\eta|_{\mathbb{H}}$ is the distance defined by \eqref{HDO}. 

The function ${\Phi\in C^2(\mathbb{R}_+)}$ is the standard cut-off function given by
\begin{equation}\label{Phi}
\Phi(z)=
 \begin{cases}
   1 &\text{if\,\,\,\, $0\leq z\leq\frac{1}{2} $},\\
   \searrow &\text{if\,\,\,\, $\frac{1}{2}<z<1$},\\
   0 &\text{if\,\,\,\, $z\geq 1$}.
 \end{cases}
\end{equation}
Furthermore, suppose that for $\displaystyle\frac{1}{q}+\frac{1}{q'}=1$ the following condition holds:
$$\int_{\mathbb{H}^n}\varphi_2^{-\frac{1}{q-1}}|\Delta_{\mathbb{H}}\varphi_2|^{q'}d\eta<+\infty.$$
\begin{lemma}\label{LT1} Let $T>0$. For sufficiently large $T, R$, we have 
\begin{equation*}\begin{split}
&\mathcal{I}_1=\int_0^T\varphi_1^{-\frac{1}{q-1}}|\varphi_1|^{q'} dt=C_1T, \\&\mathcal{I}_2=\int_0^T\varphi_1^{-\frac{1}{q-1}}|\varphi'_1|^{q'} dt=C_2T^{1-q'},    
\\&\mathcal{I}_3=\int_0^T\varphi_1^{-\frac{1}{q-1}}|\varphi''_1|^{q'} dt=C_3T^{1-2q'}
\end{split}\end{equation*}
and
\begin{equation*}\begin{split}
\mathcal{I}_4=\int_{\mathbb{H}^n}\varphi_2^{-\frac{1}{q-1}}|\Delta_{\mathbb{H}}\varphi_2|^{q'}d\eta\leq C_4R^{Q-2q'}.\end{split}\end{equation*}
\end{lemma}
\begin{proof}[Proof of Lemma \ref{LT1}.]
In view of \eqref{test1} with $\ell>\frac{q+1}{q-1},$ we arrive at
\begin{equation*}\begin{split}
&\mathcal{I}_1=\int_0^T\varphi_1^{-\frac{1}{q-1}}|\varphi_1|^{q'}dt=\int_0^T\left(1-\frac{t}{T}\right)^\ell dt=C_1T,
\\&\mathcal{I}_2=\int_0^T\varphi_1^{-\frac{1}{q-1}}|\varphi'_1|^{q'} dt=\int_0^T\left(1-\frac{t}{T}\right)^{-\frac{\ell}{q-1}}\left|\ell T^{-1}\left(1-\frac{t}{T}\right)^{\ell-1}\right|^{q'} dt=C_2T^{1-q'},    
\\&\mathcal{I}_3=\int_0^T\varphi_1^{-\frac{1}{q-1}}|\varphi''_1|^{q'} dt=\int_0^T \left(1-\frac{t}{T}\right)^{-\frac{\ell}{q-1}}\left|\frac{\ell(\ell-1)}{T^{2}}\left(1-\frac{t}{T}\right)^{\ell-2}\right|^{q'} dt=C_3T^{1-2q'},
\end{split}\end{equation*}where $C_1, C_2$ and $C_3$ are positive coefficients that are determined by
\begin{equation*}\begin{split}
 C_1=\frac{1}{\ell+1},\,\,\, C_2=\frac{(q-1)\ell^{\frac{q}{q-1}}}{\ell(q-1)-1},\,\,\,C_3=\frac{(q-1)(\ell(\ell-1))^\frac{q}{q-1}}{\ell(q-1)-q-1}. \end{split}\end{equation*}

Furthermore, because of \eqref{DIS} and \eqref{Phi}, it follows that
\begin{equation}\label{DIS1}\begin{split}
\Delta_{\mathbb{H}}\varphi_2&=\frac{|x|^2+|y|^2}{r^2}\biggl(\frac{4r^2}{R^4}\Phi''\left(\frac{r^2}{R^2}\right)+ \frac{(Q-1)2+2}{R^2}\Phi'\left(\frac{r^2}{R^2}\right)\biggr)
\\&=\frac{|x|^2+|y|^2}{r^2}\biggl(\frac{4}{R^2}\Phi''\left(\frac{r^2}{R^2}\right)+ \frac{2Q}{R^2}\Phi'\left(\frac{r^2}{R^2}\right)\biggr)
\\&\leq \frac{C}{R^{2}}\biggl(\left|\Phi''\left(\frac{r^2}{R^2}\right)\right|+\left|\Phi'\left(\frac{r^2}{R^2}\right)\right|\biggr),\end{split}\end{equation} thanks to the estimate $|{x}|^2+|{y}|^2\leq \left({\tau}^2+\left(|{x}|^2+|{y}|^2\right)^2\right)^\frac{1}{2}={r}^2$.

Hence, applying the change of variables by 
\begin{equation*}
\Tilde{x}=\frac{x}{R}, \Tilde{y}=\frac{y}{R}, \Tilde{\tau}=\frac{\tau}{R^{2}},\,\,\,\Tilde{\eta}=(\Tilde{x}, \Tilde{y}, \Tilde{\tau}), 
\end{equation*} and recalling the last estimate, we obtain
\begin{equation*}\begin{split}
\mathcal{I}_4&=\int_{\mathbb{H}^n}\varphi_2^{-\frac{1}{q-1}}|\Delta_{\mathbb{H}}\varphi_2|^{q'}d\eta
\\&\leq R^{Q-2q'}\int_{\frac{1}{2}\leq|\Tilde{\eta}|_{\mathbb{H}}^2\leq 1}\Phi^{-\frac{1}{q-1}}(\Tilde{\eta})||\Phi''(\Tilde{\eta})|+|\Phi'(\Tilde{\eta})||^{q'}d\Tilde{\eta}\\&=C_4R^{Q-2q'},  
\end{split}\end{equation*}where
$$C_4=\int_{\frac{1}{2}\leq|\Tilde{\eta}|_{\mathbb{H}}^2\leq 1}\Phi^{-\frac{1}{q-1}}(\Tilde{\eta})||\Phi''(\Tilde{\eta})|+|\Phi'(\Tilde{\eta})||^{q'}d\Tilde{\eta},$$which completes the proof.\end{proof}

Next, we consider another test function, for $0<T, R<\infty$ sufficiently large, 
\begin{equation}\label{test}
\psi(t,\eta)=\psi_1(t)\psi_2(\eta)=\left(1-\frac{t}{T}\right)^\ell\Psi^\kappa\left(\frac{\ln\left(\frac{|\eta|_{\mathbb{H}}}{\sqrt{R}}\right)}{\ln\left(\sqrt{R}\right)}\right),\,\,\,t>0,\,\eta\in\mathbb{H}^n,\end{equation} where $\ell>\frac{q+1}{q-1}$ and $\kappa>\frac{2q}{q-1}.$
\\Let the function $\Psi\in C^2(\mathbb{R})$ be a function $\Psi:\mathbb{R}\to[0,1]$, which is the standard cut-off function defined by
\begin{equation}\label{TP}
\Psi(z)=
 \begin{cases}
   1 &\text{if $-\infty<z\leq0 $},\\
   \searrow &\text{if $0<z<1$},\\
   0 &\text{if $z\geq 1$}, 
 \end{cases}\end{equation}with
  \begin{equation}\label{TPQ}
  |\Psi'(z)|\leq C,\,\,\,|\Psi''(z)|\leq C.   
 \end{equation}

\begin{lemma}\label{LT3} Let $T>0$ and $q=\displaystyle\frac{Q}{Q-2}$. For sufficiently large $T, R$, we have 
\begin{equation*}\begin{split}
&\mathcal{J}_1=\int_0^T\int_{\mathbb{H}^n}\psi^{-\frac{1}{q-1}}|\Delta_{\mathbb{H}}\psi|^{\frac{q}{q-1}} dt=C_1T\left[\left(\ln R\right)^{-Q}+\left(\ln R\right)^{-\frac{Q}{2}}\right], \\&\mathcal{J}_2=\int_0^T\int_{\mathbb{H}^n}\psi^{-\frac{1}{q-1}}|\Delta_{\mathbb{H}}\psi_t|^{\frac{q}{q-1}} dt=C_2 T^{\frac{2-Q}{2}}\left[\left(\ln R\right)^{-Q}+\left(\ln R\right)^{-\frac{Q}{2}}\right],    
\\&\mathcal{J}_3=\int_0^T\int_{\mathbb{H}^n}\psi^{-\frac{1}{q-1}}|\Delta_{\mathbb{H}}\psi_{tt}|^{\frac{q}{q-1}} dt=C_3T^{1-Q}\left[\left(\ln R\right)^{-Q}+\left(\ln R\right)^{-\frac{Q}{2}}\right].
\end{split}\end{equation*}
\end{lemma}
\begin{proof}[Proof of Lemma \ref{LT3}.] 
Thanks to the following identity
\begin{equation*}\begin{split}
\biggl[\frac{d^2}{dr^2}+\frac{Q-1}{r}\frac{d}{dr}\biggr]\psi_2
&=\frac{\kappa(\kappa-1)}{r^2\ln^2 \sqrt R }\Psi^{\kappa-2}\left(\frac{\ln\left(\frac{r}{\sqrt{R}}\right)}{\ln\left(\sqrt{R}\right)}\right)\left[\Psi'\left(\frac{\ln\left(\frac{r}{\sqrt{R}}\right)}{\ln\left(\sqrt{R}\right)}\right)\right]^2
\\&+\frac{\kappa}{r^2\ln^2 \sqrt R }\Psi^{\kappa-1}\left(\frac{\ln\left(\frac{r}{\sqrt{R}}\right)}{\ln\left(\sqrt{R}\right)}\right)\Psi''\left(\frac{\ln\left(\frac{r}{\sqrt{R}}\right)}{\ln\left(\sqrt{R}\right)}\right)
\\&+\frac{\kappa(Q-2)}{r^2\ln \sqrt R }\Psi^{\kappa-1}\left(\frac{\ln\left(\frac{r}{\sqrt{R}}\right)}{\ln\left(\sqrt{R}\right)}\right)\Psi'\left(\frac{\ln\left(\frac{r}{\sqrt{R}}\right)}{\ln\left(\sqrt{R}\right)}\right)\end{split}\end{equation*}
and noting \eqref{DIS} for the function $\Psi$ with \eqref{TPQ} we have
\begin{equation*}\begin{split}
|\Delta_{\mathbb{H}}\psi_2|&\leq\frac{(|x|^2+|y|^2)}{r^2}\frac{C}{r^2}\left[\frac{1}{\ln^2\sqrt{R}}\Psi^{\kappa-2}\left(\frac{\ln\left(\frac{r}{\sqrt{R}}\right)}{\ln\left(\sqrt{R}\right)}\right)+\frac{1}{\ln\sqrt{R}}\Psi^{\kappa-1}\left(\frac{\ln\left(\frac{r}{\sqrt{R}}\right)}{\ln\left(\sqrt{R}\right)}\right)\right]
\\&\leq \frac{C}{r^2}\left[\frac{1}{\ln^2\sqrt{R}}\Psi^{\kappa-2}\left(\frac{\ln\left(\frac{r}{\sqrt{R}}\right)}{\ln\left(\sqrt{R}\right)}\right)+\frac{1}{\ln\sqrt{R}}\Psi^{\kappa-1}\left(\frac{\ln\left(\frac{r}{\sqrt{R}}\right)}{\ln\left(\sqrt{R}\right)}\right)\right],\end{split}\end{equation*}here we have used $|{x}|^2+|{y}|^2\leq \left({\tau}^2+\left(|{x}|^2+|{y}|^2\right)^2\right)^\frac{1}{2}={r}^2$.
\\At this stage, using Schwarz inequality and
\begin{equation*}
(a+b)^m\leq 2^{m-1}(a^m+b^m),\,\,\,a\geq0,\,b\geq0,\,m=\frac{q}{q-1},\end{equation*} we have
\begin{equation}\label{L3}\begin{split}
\int_{\mathbb{H}^n}\psi_2^{-\frac{1}{q-1}}|\Delta_\mathbb{H}\psi_2|^{q'}d\eta&\leq C\int_{\mathbb{H}^n}\psi_2^{-\frac{1}{q-1}}\left[\frac{1}{|\eta|_{\mathbb{H}}^2\ln^2\sqrt{R}}\left|\psi_2^{\frac{\kappa-2}{\kappa}}\right|\right]^\frac{q}{q-1}   d\eta
\\&+C\int_{\mathbb{H}^n}\psi_2^{-\frac{1}{q-1}}\left[\frac{1}{|\eta|_{\mathbb{H}}^2\ln\sqrt{R}}\left|\psi_2^{\frac{\kappa-1}{\kappa}}\right|\right]^\frac{q}{q-1}   d\eta. \end{split}
\end{equation}
Therefore, noting that $\kappa>\frac{2q}{q-1}$ and \eqref{TP}, we have
\begin{equation*}\begin{split}
\int_{\mathbb{H}^n}\psi_2^{-\frac{1}{q-1}}|\Delta_\mathbb{H}\psi_2|^{q'}d\eta&\leq C\int_{\mathbb{H}^n}\left[\frac{1}{|\eta|_{\mathbb{H}}^2\ln^2\sqrt{R}}\right]^\frac{q}{q-1}   d\eta
\\&+C\int_{\mathbb{H}^n} \left[\frac{1}{|\eta|_{\mathbb{H}}^2\ln\sqrt{R}}\right]^\frac{q}{q-1}   d\eta. \end{split}
\end{equation*}
Since $q=\displaystyle\frac{Q}{Q-2}$, it follows that
\begin{equation}\label{R1}\begin{split}
\int_{\mathbb{H}^n}\left[\frac{1}{|\eta|_{\mathbb{H}}^2\ln^2\sqrt{R}}\right]^\frac{q}{q-1}   d\eta&=\left(\ln R\right)^{-\frac{2q}{q-1}}\int\limits_{1<|\eta|_{\mathbb{H}}<R}|\eta|_{\mathbb{H}}^{-\frac{2q}{q-1}}   d\eta
\\&\stackrel{|\eta|_{\mathbb{H}}=r}{=}\left(\ln R\right)^{-\frac{2q}{q-1}}\int\limits_{1<r<R}r^{-\frac{2q}{q-1}+Q-1}   dr
\\&\leq C\left(\ln R\right)^{-\frac{2q}{q-1}}R^{-\frac{2q}{q-1}+Q}
\\&\leq C\left(\ln R\right)^{-Q},\end{split}\end{equation} since $\frac{2q}{q-1}=\frac{\frac{2Q}{Q-2}}{\frac{2}{Q-2}}=Q.$
\\Similarly, we deduce that
\begin{equation}\label{R2}\begin{split}
\int_{\mathbb{H}^n} \left[\frac{1}{|\eta|_{\mathbb{H}}^2\ln\sqrt{R}}\right]^\frac{q}{q-1}   d\eta&\leq C\left(\ln R\right)^{-\frac{Q}{2}}. \end{split}
\end{equation}

Finally, combining the results of Lemma \ref{LT1} with \eqref{R1}-\eqref{R2}, we obtain the desired estimates for $\mathcal{J}_1, \mathcal{J}_2$ and $\mathcal{J}_3,$ which complete the proof.
\end{proof}

\section{The proof of main results}
In this section, we derive the main results of this work.
\begin{definition}\label{WS} Let $u_0\in L^q(\mathbb{H}^n).$ A function $u\in L^q_{\text{loc}}([0,T]; L^q_{\text{loc}}(\mathbb{H}^n))$ is called a local weak solution to the problem \eqref{1}, if
\begin{equation}\label{WSF} \begin{split}&\int_0^T\int_{\mathbb{H}^n} (|u|^q\varphi +u\Delta_{\mathbb{H}}\varphi-u\Delta_{\mathbb{H}}\varphi_t)d\eta dt=\int_{\mathbb{H}^n}u_0\Delta_{\mathbb{H}}\varphi(0,\eta)d\eta, \end{split}\end{equation} 
holds for any $\varphi\in C_{t,\eta}^{1,2}([0,T];\mathbb{H}^n)$ with $\text{supp}(\varphi)\subset [0,T]\times \mathbb{H}^n$, $\varphi(T,\eta)=0$. \end{definition}

\begin{definition}\label{WS1} Let $u_0,u_1\in L^q(\mathbb{H}^n).$ A function $u\in L^q_{\text{loc}}([0,T]; L^q_{\text{loc}}(\mathbb{H}^n))$ is called a local weak solution of \eqref{2}, if
\begin{equation}\label{WSF1} \begin{split}\int_0^T\int_{\mathbb{H}^n} (|u|^q\varphi +u\Delta_{\mathbb{H}}\varphi+u\Delta_{\mathbb{H}}\varphi_{tt})d\eta dt&=\int_{\mathbb{H}^n}u_1\Delta_{\mathbb{H}}\varphi(0,\eta)d\eta\\&-\int_{\mathbb{H}^n}u_0\Delta_{\mathbb{H}}\varphi_t(0,\eta)d\eta,\end{split}\end{equation} 
holds for any $\varphi\in C_{t,\eta}^{2,2}([0,T];\mathbb{H}^n)$ with $\text{supp}(\varphi)\subset [0,T]\times \mathbb{H}^n$, $\varphi(T,\eta)=\varphi_t(T,\eta)=0$. \end{definition}

\begin{proof}[Proof of Theorem \ref{T1}.] We will separately consider subcritical and critical cases.

$\bullet$ $ \textbf{Subcritical case $\displaystyle q<\frac{Q}{Q-2}$.}$
Applying H\"{o}lder's inequality in the left side of \eqref{WSF}, we obtain
\begin{equation*}\begin{split}
\int_0^T\int_{\mathbb{H}^n}u\Delta_{\mathbb{H}}\varphi_td\eta dt&\leq\int_0^T\int_{\mathbb{H}^n}|u|\varphi^{\frac{1}{q}}\varphi^{-\frac{1}{q}}|\Delta_{\mathbb{H}}\varphi_t|d\eta dt
\\&\leq\biggl(\int_0^T\int_{\mathbb{H}^n}|u|^q\varphi d\eta dt\biggr)^\frac{1}{q} \biggl(\int_0^T\int_{\mathbb{H}^n}\varphi^{-\frac{1}{q-1}}|\Delta_{\mathbb{H}}\varphi_t|^{q'}d\eta dt\biggr)^\frac{1}{q'}
\end{split}\end{equation*}
and
\begin{equation*}\begin{split}
\int_0^T\int_{\mathbb{H}^n}u\Delta_{\mathbb{H}}\varphi d\eta dt&\leq\int_0^T\int_{\mathbb{H}^n}|u|\varphi^{\frac{1}{q}}\varphi^{-\frac{1}{q}}|\Delta_{\mathbb{H}}\varphi|d\eta dt
\\&\leq\biggl(\int_0^T\int_{\mathbb{H}^n}|u|^q\varphi d\eta dt\biggr)^\frac{1}{q}\biggl(\int_0^T\int_{\mathbb{H}^n}\varphi^{-\frac{1}{q-1}}|\Delta_{\mathbb{H}}\varphi|^{q'}d\eta dt\biggr)^\frac{1}{q'}.
\end{split}\end{equation*}
Therefore, using the $\varepsilon$-Young inequality
$$XY\leq \frac{\varepsilon}{q} X^q+\frac{1}{q'\varepsilon^{q'-1}}Y^{q'},\,\, \frac{1}{q}+\frac{1}{q'}=1,\,\, X,Y,\varepsilon>0,$$
in the right-hand side of last inequalities with $\displaystyle\varepsilon=\frac{q}{4}$, we can rewrite \eqref{WSF} in the following form
\begin{equation}\label{M3}\begin{split}
\int_0^T\int_{\mathbb{H}^n}|u|^q\varphi d\eta dt&\leq 2C(q) \biggl(\int_0^T\int_{\mathbb{H}^n}\varphi^{-\frac{1}{q-1}}|\Delta_{\mathbb{H}}\varphi_t|^{q'}d\eta dt\biggr)
\\&+2C(q)\biggl(\int_0^T\int_{\mathbb{H}^n}\varphi^{-\frac{1}{q-1}}|\Delta_{\mathbb{H}}\varphi|^{q'}d\eta dt\biggr)
\\&+2\int_{\mathbb{H}^n}u_0\Delta_{\mathbb{H}}\varphi_2d\eta,
\end{split}\end{equation}where $\displaystyle C(q)=\biggl(\frac{q}{4}\biggr)^{1-q'}\frac{1}{q'}.$
\\Choosing the function $\varphi$ as in \eqref{test1} we arrive at
\begin{equation*}\begin{split}
&\int_0^T\int_{\mathbb{H}^n}\varphi^{-\frac{1}{q-1}}|\Delta_{\mathbb{H}}\varphi_t|d\eta dt=\underbrace{\biggl(\int_0^T\varphi_1^{-\frac{1}{q-1}}|\varphi'_1|^{q'} dt\biggr)}_{\mathcal{I}_2}\underbrace{\biggl(\int_{\mathbb{H}^n}\varphi_2^{-\frac{1}{q-1}}|\Delta_{\mathbb{H}}\varphi_2|^{q'}d\eta \biggr)}_{\mathcal{I}_4},
\\&\int_0^T\int_{\mathbb{H}^n}\varphi^{-\frac{1}{q-1}}|\Delta_{\mathbb{H}}\varphi|d\eta dt=\underbrace{\biggl(\int_0^T\varphi_1^{-\frac{1}{q-1}}|\varphi_1|^{q'}dt\biggr)}_{\mathcal{I}_1}\underbrace{\biggl(\int_{\mathbb{H}^n}\varphi_2^{-\frac{1}{q-1}}|\Delta_{\mathbb{H}}\varphi_2|^{q'}d\eta \biggr)}_{\mathcal{I}_4}.
\end{split}\end{equation*}

Using Lemma \ref{LT1}, we deduce that
\begin{equation*}\begin{split}
\int_0^T\int_{\mathbb{H}^n}|u|^q\varphi_1\varphi_2 d\eta dt&\leq C_5(q)T^{1-q'}R^{Q-2q'}+C_6(q)TR^{Q-2q'}
\\&+2\int_{\mathbb{H}^n}u_0\Delta_{\mathbb{H}}\varphi_2d\eta,
\end{split}\end{equation*} where $C_5(q)=2C(q)C_1C_4$ and $C_6(q)=2C(q)C_2C_4,$ respectively.

Since $u_0\in L^q(\mathbb{H}^n)$ and using the H\"{o}lder inequality in the last term, and taking account of \eqref{DIS1}, we obtain
\begin{equation*}\begin{split}
\left|\int_{\mathbb{H}^n}u_0\Delta_{\mathbb{H}}\varphi_2d\eta\right| &\leq \biggl(\int_{\mathbb{H}^n}|u_0|^qd\eta\biggr)^\frac{1}{q}\biggl(\int_{\mathbb{H}^n}|\Delta_{\mathbb{H}}\varphi_2|^{q'}d\eta\biggl)^\frac{1}{q'},\end{split}\end{equation*} with
\begin{equation}\label{O1}
\int_{\mathbb{H}^n}|\Delta_{\mathbb{H}}\varphi_2|^{q'}d\eta\leq R^{Q-2q'}\int_{\frac{1}{2}\leq|\Tilde{\eta}|_{\mathbb{H}}^2\leq1}|\Phi''(\Tilde{\eta})|+|\Phi'(\Tilde{\eta})||^{q'}d\Tilde{\eta}\leq C R^{Q-2q'},\end{equation}where $C$ does not depend on $R>0$.

As a result, we arrive at
\begin{equation*}\begin{split}
\int_0^T\int_{\mathbb{H}^n}|u|^q\varphi_1\varphi_2 d\eta dt&\leq C_5(q)T^{1-q'}R^{Q-2q'}+C_6(q)TR^{Q-2q'}
\\&+C R^{Q-2q'}\biggl(\int_{\mathbb{H}^n}|u_0|^qd\eta\biggr)^\frac{1}{q}.
\end{split}\end{equation*}
Finally, passing $R\to+\infty$ in the last inequality for $Q-2q'<0$, we deduce that
\begin{equation*}\begin{split}
\int_0^T\int_{\mathbb{H}^n}|u|^q\varphi_1(t)d\eta dt&\leq 0\,\,\,\text{for all}\,\,\,T>0.
\end{split}\end{equation*}  
Hence, we have $u(t, \eta) = 0$ almost everywhere for $\eta \in \mathbb{H}^N$ and $t \in [0, T]$. Since the time $T > 0$ is arbitrary in this case, we come to the conclusion that nontrivial local weak solutions do not exist, i.e., as a result, we come to the conclusion about instantaneous blowup of the solutions.

$\bullet$ {\bf Critical case $q=\displaystyle\frac{Q}{Q-2}$}. At this stage, we choose the test function as in \eqref{test}. Therefore, we will derive the following inequality instead of \eqref{M3}, using the same procedure as in the previous case 
\begin{equation}\label{M4}\begin{split}
\int_0^T\int_{\mathbb{H}^n}|u|^q\psi d\eta dt&\leq 2C(q) \underbrace{\biggl(\int_0^T\int_{\mathbb{H}^n}\psi^{-\frac{1}{q-1}}|\Delta_{\mathbb{H}}\psi_t|^{q'}d\eta dt\biggr)}_{\mathcal{J}_2}
\\&+2C(q)\underbrace{\biggl(\int_0^T\int_{\mathbb{H}^n}\psi^{-\frac{1}{q-1}}|\Delta_{\mathbb{H}}\psi|^{q'}d\eta dt\biggr)}_{\mathcal{J}_1}
\\&+2\int_{\mathbb{H}^n}u_0\Delta_{\mathbb{H}}\psi_2d\eta.
\end{split}\end{equation}
Applying H\"{o}lder's inequality in the last term of \eqref{M4},
\begin{equation*}\begin{split}
\left|\int_{\mathbb{H}^n}u_0\Delta_{\mathbb{H}}\psi_2d\eta\right| &\leq \biggl(\int_{\mathbb{H}^n}|u_0|^qd\eta\biggr)^\frac{1}{q}\biggl(\int_{\mathbb{H}^n}|\Delta_{\mathbb{H}}\psi_2|^{q'}d\eta\biggl)^\frac{1}{q'}\end{split}\end{equation*}and recalling \eqref{L3} with $\kappa>\frac{2q}{q-1}>2$, one obtains
\begin{equation*}\begin{split}
\int_{\mathbb{H}^n}|\Delta_{\mathbb{H}}\psi_2|^{q'}d\eta\leq C\left[\left(\ln R\right)^{-Q}+\left(\ln R\right)^{-\frac{Q}{2}}\right],     
\end{split}\end{equation*}where $C$ is a constant.

Finally, using Lemma \ref{LT3} and the last inequality in view of $u_0\in L^q(\mathbb{H}^n)$, we have
\begin{equation}\label{MM5}\begin{split}
\int_0^T\int_{\mathbb{H}^n}|u|^q\psi_1\psi_2 d\eta dt&\leq \left[\mathcal{C}_1(q)T^{\frac{2-Q}{2}}+\mathcal{C}_2(q)T+2C\right]\left[\left(\ln R\right)^{-Q}+\left(\ln R\right)^{-\frac{Q}{2}}\right],
\end{split}\end{equation}where $\mathcal{C}_1(q)=2{C}(q){C}_1$ and $\mathcal{C}_2(q)=2{C}(q){C}_2.$ 
\\Consequently, passing to the limit as $R\to+\infty$ in \eqref{MM5}, we obtain
\begin{equation*}
\int_0^T\int_{\mathbb{H}^n}|u|^q\psi_1 d\eta dt\leq 0\,\,\,\text{for all}\,\,\,T>0,\end{equation*}
which proves the absence of a local weak solution of \eqref{1}.
\end{proof}
\begin{proof}[Proof of Theorem \ref{T2}.] $\bullet$ $ \textbf{Subcritical case $\displaystyle q<\frac{Q}{Q-2}$.}$ In this case,  we choose the same test function as in the first part of Theorem \ref{T1}.
\\Repeating the procedure above, we obtain the inequality
\begin{equation}\label{Q1}\begin{split}
\int_0^T\int_{\mathbb{H}^n}|u|^q\varphi_1\varphi_2 d\eta dt&\leq 2C(q) \underbrace{\biggl(\int_0^T\varphi_1^{-\frac{1}{q-1}}|\varphi''_1|^{q'} dt\biggr)}_{\mathcal{I}_3}\underbrace{\biggl(\int_{\mathbb{H}^n}\varphi_2^{-\frac{1}{q-1}}|\Delta_{\mathbb{H}}\varphi_2|^{q'}d\eta \biggr)}_{\mathcal{I}_4}
\\&+2C(q)\underbrace{\biggl(\int_0^T\varphi_1^{-\frac{1}{q-1}}|\varphi_1|^{q'}  dt\biggr)}_{\mathcal{I}_1}\underbrace{\biggl(\int_{\mathbb{H}^n}\varphi_2^{-\frac{1}{q-1}}|\Delta_{\mathbb{H}}\varphi_2|^{q'}d\eta \biggr)}_{\mathcal{I}_4}
\\&+2\int_{\mathbb{H}^n}\biggl(u_1\Delta_{\mathbb{H}}\varphi(0,\eta)-u_0\Delta_{\mathbb{H}}\varphi_t(0,\eta)\biggr)d\eta,
\end{split}\end{equation}where $\displaystyle C(q)=\biggl(\frac{q}{4}\biggr)^{1-q'}\frac{1}{q'}.$

Consequently, from Lemma \ref{LT1} and using the H\"{o}lder inequality in the last term of \eqref{Q1} with $u_0,u_1\in L^q(\mathbb{H}^n)$ , we get
\begin{equation*}\begin{split}
&\left|\int_{\mathbb{H}^n}u_1\Delta_{\mathbb{H}}\varphi(0,\eta)d\eta\right| \leq \biggl(\int_{\mathbb{H}^n}|u_1|^qd\eta\biggr)^\frac{1}{q}\biggl(\int_{\mathbb{H}^n}|\Delta_{\mathbb{H}}\varphi_2|^{q'}d\eta\biggl)^\frac{1}{q'},
\\&\left|\int_{\mathbb{H}^n}u_0\Delta_{\mathbb{H}}\varphi_t(0,\eta)\right| \leq T^{-1} \biggl(\int_{\mathbb{H}^n}|u_0|^qd\eta\biggr)^\frac{1}{q}\biggl(\int_{\mathbb{H}^n}|\Delta_{\mathbb{H}}\varphi_2|^{q'}d\eta\biggl)^\frac{1}{q'}.
\end{split}\end{equation*}
Hence, using the \eqref{O1}, we arrive at the following.
\begin{equation*}\begin{split}
\int_0^T\int_{\mathbb{H}^n}|u|^q\varphi_1\varphi_2 d\eta dt&\leq C R^{Q-2q'}\left(T^{1-2q'}+T+1+T^{-1}\right),\end{split}\end{equation*} where $C>0$ is the arbitrary constant independent of $T$ and $R.$

Finally, passing $R\to+\infty$ into the last inequality for $Q-2q'<0$, we have
\begin{equation*}\begin{split}
\int_0^T\int_{\mathbb{H}^n}|u|^q\varphi_1d\eta dt&\leq 0\,\,\,\text{for all}\,\,\,T>0,
\end{split}\end{equation*}which proves that the problem \eqref{2} has no local weak solution for any $T>0$.

$\bullet$ {\bf Critical case $q=\displaystyle\frac{Q}{Q-2}$}. Choosing the same test function as in the second case of Theorem \ref{T1} and repeating the above technique, we obtain the following result
\begin{equation}\label{QQ1}\begin{split}
\int_0^T\int_{\mathbb{H}^n}|u|^q\psi_1\psi_2 d\eta dt&\leq 2C(q) \underbrace{\biggl(\int_0^T\int_{\mathbb{H}^n}\psi^{-\frac{1}{q-1}}|\Delta_{\mathbb{H}}\psi_{tt}|^{q'}d\eta dt\biggr)}_{\mathcal{J}_3}
\\&+2C(q)\underbrace{\biggl(\int_0^T\int_{\mathbb{H}^n}\psi^{-\frac{1}{q-1}}|\Delta_{\mathbb{H}}\psi|^{q'}d\eta dt\biggr)}_{\mathcal{J}_1}
\\&+2\int_{\mathbb{H}^n}\biggl(u_1\Delta_{\mathbb{H}}\psi(0,\eta)-u_0\Delta_{\mathbb{H}}\psi_t(0,\eta)\biggr)d\eta,
\end{split}\end{equation}where $\displaystyle C(q)=\biggl(\frac{q}{4}\biggr)^{1-q'}\frac{1}{q'}.$

Moreover, taking into account the result of  Lemma \ref{LT3}, recalling $u_0,u_1\in L^q(\mathbb{H}^n)$ and using the H\"{o}lder inequality in the last term of \eqref{QQ1}, we have
\begin{equation*}\begin{split}
\left|\int_{\mathbb{H}^n}u_1\Delta_{\mathbb{H}}\psi(0,\eta)d\eta\right| \leq \biggl(\int_{\mathbb{H}^n}|u_1|^qd\eta\biggr)^\frac{1}{q}\biggl(\int_{\mathbb{H}^n}|\Delta_{\mathbb{H}}\psi_2|^{q'}d\eta\biggl)^\frac{1}{q'},\end{split}\end{equation*}
and
\begin{equation*}\begin{split}\left|\int_{\mathbb{H}^n}u_0\Delta_{\mathbb{H}}\psi_t(0,\eta)\right| \leq T^{-1}\biggl(\int_{\mathbb{H}^n}|u_0|^qd\eta\biggr)^\frac{1}{q}\biggl(\int_{\mathbb{H}^n}|\Delta_{\mathbb{H}}\psi_2|^{q'}d\eta\biggl)^\frac{1}{q'}.
\end{split}\end{equation*}
Furthermore, in view of Lemma \ref{LT3} with $\kappa>2q$ we have
\begin{align*}
\int_{\mathbb{H}^n}|\Delta_{\mathbb{H}}\psi_2|^{q'}d\eta&\leq C\left[\left(\ln R\right)^{-Q}+\left(\ln R\right)^{-\frac{Q}{2}}\right].\end{align*}
Consequently, we have the following inequality
\begin{equation*}\begin{split}
&\int_0^T\int_{\mathbb{H}^n}|u|^q\psi_1\psi_2 d\eta dt
\leq C\left[T^{1-Q}+T+1+T^{-1}\right]
\left[\left(\ln R\right)^{-Q}+\left(\ln R\right)^{-\frac{Q}{2}}\right],
\end{split}\end{equation*} where $C>0$ is the arbitrary constant independent of $T$ and $R.$

Then, passing $R\to+\infty$, we deduce that
\begin{equation*}\label{M8}\begin{split}
\int_0^T\int_{\mathbb{H}^n}|u|^q\psi_1d\eta dt&\leq 0\,\,\,\text{for all}\,\,\,T>0,
\end{split}\end{equation*}which proves that \eqref{2} has no local weak solution for any $T>0$. 
\end{proof}

\section*{Declaration of competing interest}
	The authors declare that there is no conflict of interest.

\section*{Data Availability Statements} The manuscript has no associated data.

\section*{Acknowledgments}
This research has been funded by the Science Committee of the Ministry of Education and Science of the Republic of Kazakhstan (Grant No. AP14869090), by the FWO Odysseus 1 grant G.0H94.18N: Analysis and Partial Differential Equations and by the Methusalem programme of the Ghent University Special Research Fund (BOF) (Grant number 01M01021). Michael Ruzhansky is also supported by EPSRC grant EP/R003025/2 and EP/V005529/1.

\end{document}